\documentclass[a4paper,11pt]{article}

\usepackage{amsmath,amssymb,amsthm,amsfonts,setspace,hyperref,dsfont}
\usepackage{latexsym}
\usepackage{amsmath,amsthm}
\usepackage{amssymb}
\usepackage{multicol}
\usepackage{graphics}
\usepackage{graphicx}

\newtheorem{theorem}{Theorem}[section]

\newtheorem{lemma}[theorem]{Lemma}

\def\eqref#1{(\ref{#1})}

\def\<<{\prec}

\begin{document}

\title{{
\bf Sensitive open map semigroups on Peano continua having a free arc }\footnotetext{*Corresponding author.}
\author{Suhua Wang$^{a,*}$, Enhui Shi$^{b}$, Pingping Dong$^{c}$ \vspace{3mm}\\
\small {\sl $^a$Zhangjiagang Campus, } \\
\small{\sl Jiangsu University of Science and Technology,}\\
\small{\sl Zhangjiagang, Jiangsu 215600, P.R.China}\\
\small{ E-mail: wangsuhuajust@163.com}\vspace{2mm}\\
\small {\sl $^b$Department of Mathematics,}\\
\small {\sl Soochow University, Suzhou, Jiangsu 215006,
P.R.China}\\
\small{ E-mail: ehshi@suda.edu.cn}\vspace{2mm}\\
\small {\sl $^c$Department of Mathematics,}\\
\small {\sl Soochow University, Suzhou, Jiangsu 215006,
P.R.China}\\
\small{ E-mail: 1690475436@qq.com}
}
}
\date{} \maketitle

{\narrower\small\noindent{\bf Abstract:}\quad  Let $X$ be a Peano continuum having a free arc and let $C^0(X)$ be the semigroup of continuous self-maps of $X$. A subsemigroup $F\subset C^0(X)$ is said to be sensitive, if there is
some constant $c>0$ such that for any nonempty open set $U\subset
X$, there is some $f\in F$ such that the diameter ${\rm
diam}(f(U))>c$. We show that if $X$ admits a sensitive commutative subsemigroup $F$ of $C^0(X)$ consisting of continuous open maps, then either $X$ is an arc, or $X$ is a circle.\vspace{2mm}



}\vspace{5mm}

\section{Introduction}
\setlength{\parskip}{1mm}

Let $X$ be a compact metric space with metric $d$, and let $C^0(X)$ be the set
of all continuous self-maps of $X$. Then $C^0(X)$ is a semigroup
under the composition of maps. If $F$ is a subsemigroup of
$C^0(X)$, then we call the pair $(X, F)$ a  {\it dynamical system}
on $X$. If $F$ is generated by a single map $f$, that is $F=\{f^n:
n=0, 1, 2, \cdots\}$, then we call $(X, F)$ a {$\mathbb Z_+$-\it
dynamical system} which is also denoted by $(X, f)$. If there is
some constant $c>0$ such that for any nonempty open set $U\subset
X$, there is some $f\in F$ such that the diameter ${\rm
diam}(f(U))>c$, then $F$ is said to be {\it sensitive} or {\it
$c$-sensitive}, and $c$ is said to be a {\it sensitivity constant}
for $F$. For any $f\in C^0(X)$, if the semigroup $F=\{f^n: n=0, 1,
2, \cdots\}$ is sensitive, then $f$ is called {\it sensitive}.

Sensitivity plays an important role in chaos theory. Although there is no universally accepted mathematical definition of chaos, sensitivity is widely understood as being the central idea of chaos. For example, in 1980 Auslander and Yorke introduced their definition of chaos by associating sensitivity with topological transitivity (see \cite{AY}). Later Devaney introduced a famous definition of chaos called Devaney chaos, in which a system is chaotic if it is sensitive, topologically transitive and has a dense set of periodic points (see \cite{De}). Sensitivity has been extensively studied by many authors(see e.g. \cite{AK,BBC,Bl,GW,Gu,HY,LTY,Su,YY}).

By a {\it continuum}, we mean a compact connected metric space. An {\it arc} is a continuum homeomorphic to the closed interval $[0,1]$ and a {\it circle} is a continuum homeomorphic to the unit circle ${\mathbb S}^1=\{e^{{\bf i}\theta}:0\leq \theta \leq 2\pi\}$ in the complex plane. If a continuum is locally connected, then it
is called a {\it Peano continuum}.
Let $X$ be a metric space and let $J$ be a subset of $X$. We say
that $J$ is a {\it free open interval} of $X$ if $J$ is  open and
is homeomorphic to the open interval $(0,1)$. If the closure
$\overline J$ of a free open interval $J$ of $X$ is homeomorphic
to $[0,1]$, then $\overline J$  is called a {\it free arc} of $X$.
Notice that a space $X$ contains a free arc if and only if it
contains a free open interval.

In the study of topological dynamical systems and continuum theory, one is interested in the question: given a continuum $X$, dose it admit a sensitive subsemigroup $F\subset C^0(X)$? This depends on the topology of $X$ and the algebraic structure of $F$.  It is well known that arcs and circles admit no sensitive homeomorphism. Mai and Shi in \cite{MS} proved a stronger result: graphs admit no sensitive commutative subsemigroup of $C^{0}(X)$ consisting of homeomorphisms. In the same paper, they also gave an example of a sensitive homeomorphism on a Suslinian continuum having a free arc, which answers a question posed by Kato in \cite{Kato}. Notice that the Suslinian continuum discussed in this example is not locally connected. When we consider Peano continua, the consequence is different.
The notion of Peano continuum having a free arc is a natural generalization of arcs, circles and graphs. Early in 1988, Kawamura showed that every Peano continuum having a free arc admits no expansive homeomorphism in \cite{Ka}. Furthermore, in \cite{MS1} Mai and Shi showed that if $X$ is a Peano continuum having a free arc and $F$ is a commutative semigroup consisting of homeomorphisms on $X$, then $F$ is not sensitive. This result generalized the main results in \cite{Ka,MS}. However, it is not true if we replace the condition ``homeomorphism" by ``open map" because there exist sensitive open maps on arcs and circles. For example, the tent map on the closed interval $[0,1]$ is a sensitive open map, and the circle map $f: {\mathbb S}^1\to {\mathbb S}^1$, $e^{{\bf i}\theta}\mapsto e^{2{\bf i}\theta}$, for all $e^{{\bf i}\theta}\in {\mathbb S}^1$ is also sensitive. In fact, Shi et al. in \cite{SWM} proved that if a Peano continuum having a free arc admits a sensitive open map, then it is either an arc, or a circle, and Mai et al. in \cite{MSW} proved that if a Peano continuum having a free arc admits a sensitive commutative semigroup which consists of local homeomorphisms, then it must be a circle.
 In this paper, we will consider the further question: what can we say about the structure of $X$ provided that $X$ is a Peano continuum having a free arc and $F\subset C^0(X)$ is a sensitive commutative semigroup consisting of open maps on $X$?
We prove the following theorem:

\begin{theorem}\label{T1.1} Let $X$ be a Peano continuum having a
free arc. If $X$ admits a sensitive commutative subsemigroup $F$ of $C^0(X)$
consisting of continuous open maps, then either $X$ is
an arc, or $X$ is a circle.
\end{theorem}

This theorem generalizes the corresponding results in \cite{Ka,MSW,SWM}.

\section{The topological structure of Peano continua having a free arc}
Throughout this paper, by the symbol ${\mathbb N}$, we mean the set of positive integers. Let $X$ be a compact metric space with metric $d$. For a subset $A$ of $X$, we use the symbols $\overline{A}$ and ${\rm Bd}_X(A)$ to denote the closure and the boundary of $A$ in $X$ respectively. The cardinality of $A$ is denoted by $|A|$. Define the diameter of $A$ by ${\rm diam}(A)={\rm sup}\{d(a,b): a,b\in A\}$. If $\epsilon>0$ and $A,B\subset X$, let $D_{\epsilon}(A)=\{x\in X: d(x,a)<\epsilon {\rm \ for\ some\ }a\in A\}$ and let $d(A,B)={\rm inf}\{d(x,y):x\in A,y\in B\}$. Suppose that $X$ is a continuum, a point $x\in X$ is said to be an {\it endpoint} of $X$ provided that for each open set $U$ with $x\in U$, there exists an open set $V$ such that $x\in V\subset U$ and ${\rm Bd}_X(V)$ consists of precisely one point. The set of all endpoints of $X$ is denoted by ${\rm End}(X)$. A subset $A$ in $X$ is {\it uniquely arcwise connected} if for any $x\neq y\in A$, there is a unique arc $[x,y]_A\subset A$
connecting $x$ and $y$; denote $[x,y)_A=(y,x]_A=[x,y]_A-\{y\}$ and $(x,y)_A=[x,y]_A-\{x,y\}$.

A continuum is called {\it nondegenerate} if it has at least two points. Let us recall some properties of Peano continua. Every nondegenerate Peano continuum is arcwise connected (see \cite[Theorem 8.23]{Na}); every open subset of a Peano continuum is locally arcwise connected; every Peano continuum is locally arcwise connected (see \cite[Theorem 8.25]{Na}). The following lemma comes from \cite[Exercise 8.30]{Na}.

\begin{lemma}\label{L2.1}
Let $X$ be a Peano continuum. For each $\epsilon>0$,
there exists a $\delta=\delta(\epsilon)\in (0,\epsilon/2]$ such
that, for any $x,y\in X$ with $0<d(x,y)\leq \delta$, there always
exists an arc $A$ with endpoints $x$ and $y$ such that ${\rm
diam}(A)<\epsilon$.
\end{lemma}

Let $A$ be an arc. Fix a homeomorphism $h:A\to [0,1]$. Let $\prec$
be an ordering on $A$ which is defined by  $x\prec y$ if and only if
$h(x)<h(y)$ for all $x,y\in A$. Then $\prec$ is called a {\it natural ordering} on $A$ induced by
 $h$. The following lemma can be seen in \cite{MSW}.

\begin{lemma}\label{L2.2}
Let $X$ be a Peano continuum having a free arc and let $A$ be a
free arc of $X$ with endpoints $a$ and $b$. Suppose
$\{x,y\}\subset A$ are such that $a\prec x\prec y\prec b$ with
respect to a natural ordering $\prec$ on $A$. Let $\epsilon={\rm
min}\{{\rm diam}([a,x]_A), {\rm diam}([y,b]_A)\}$, and let
$\delta=\delta(\epsilon)$ be as in the statement of
Lemma~\ref{L2.1}, then $d([x,y]_A, X-(a,b)_A)>\delta$.
\end{lemma}

A continuum $Y$ is said to be a {\it simple troid} if it is homeomorphic to the subset $\{re^{\mathbf{i}\theta}: r\in [0,1], \theta=\pi/3, \pi,
5\pi/3\}$ of the complex plane. If $Y$ is a simple troid, then the
unique point $o\in Y$ such that $Y\setminus \{o\}$ has three
components is called the {\it center of $Y$}.

\begin{lemma}\label{L2.3}
Let $X$ be a nondegenerate Peano continuum, and let $a\in X$ be such that there is a neighborhood $U$ of $a$ which contains no simple troid . Then one of the following statements holds:

(i) If $a\in {\rm End}(X)$, then there exists a connected open neighborhood $V$ of $a$ such that $V$ is homeomorphic to the half-open interval $[0,1)$;

(ii) If $a\notin {\rm End}(X)$, then there exists a connected open neighborhood $V$ of $a$ such that $V$ is homeomorphic to the open interval $(0,1)$.
\end{lemma}

The proof of Lemma~\ref{L2.3} is similar to that of Lemma 3.1 in \cite{SWM}. For simplicity, we omit the proof here.
The next lemma follows directly from Lemma~\ref{L2.3}.

\begin{lemma}\label{L2.4}
Let $X$ be a nondegenerate Peano continuum which is neither an arc nor a circle. Suppose that $I$ is an open connected subset of $X$ which contains no simple troid. Then one of the following statements holds:

(i) $I$ is homeomorphic to the open interval $(0,1)$;

(ii) $I$ is homeomorphic to the half-open interval $[0,1)$.
\end{lemma}

A free open interval $J$ in a continuum $X$ is called {\it maximal}, if there is no free open interval in $X$ which properly includes $J$. If $X$ is a Peano continuum, then by the local connectedness of $X$, it is easy to see that there are three types for the closure of a maximal free open interval $J$ in $X$:

(i) $\overline{J}=J\cup \{a,b\}$ is an arc, where $a,b\in X\setminus
 J$ are endpoints of $\overline{J}$, and $\{a,b\}\cap {\rm End}(X)\neq\emptyset$;

(ii) $\overline{J}=J\cup \{a,b\}$ is an arc, where $a,b\in X\setminus
 J$ are endpoints of $\overline{J}$, and $\{a,b\}\cap {\rm End}(X)=\emptyset$;

(iii) $\overline{J}=J\cup \{c\}$ is a circle, for some $c\in X\setminus
 J$.

 In particular, if the closure of $J$ is the first type, then we call $J$ a {\it maximal free open interval of Type I}.

\begin{lemma}\label{L2.5}
Let $X$ be a nondegenerate Peano continuum having a free arc, which is neither an arc nor a circle. Then for any free open interval $I\subset X$, there is a unique maximal free open interval $J$ containing $I$.
\end{lemma}

\begin{proof}
Let ${\mathcal P}=\{P: P {\rm \ is\ a\ free\ open\ interval\
containing\ } I\}$. Set $J=\underset{P\in {\mathcal
P}}{\cup}P$. Then $J$ is a connected open subset of $X$. Clearly, $J$ contains no simple troid and $x\notin {\rm End}(X)$ for every $x\in J$. Thus $J$ is a free open interval by Lemma~\ref{L2.4}, and the maximality of $J$ is implied by the definition of $J$.
\end{proof}

\noindent {\bf Remark.} If $X$ is a nondegenerate Peano continuum having a free arc, which is not a circle,
then it is not difficult to see that two maximal free open intervals in $X$ either are disjoint or coincide.

\section{Open maps}
In this section, we introduce some lemmas about open maps on Peano continua. The first lemma is from \cite[Lemma 3.2]{SWM}

\begin{lemma}\label{L3.1}
Let $J$ be the open interval $(0,1)$ and let $X$ be a metric
 space. If $f:J\to X$ is an open map, then $f(J)$ contains no
 simple troid.
\end{lemma}

\begin{lemma}\label{L3.2}
 Let $X$ be a continuum, and let $f:X\to X$ be an open map. If $a$ is an endpoint of $X$, then $f(a)$ is also an endpoint of $X$.
\end{lemma}

\begin{proof}
For any open neighborhood $U$ of $f(a)$, $f^{-1}(U)$ is an open neighborhood of $a$ by the continuity of $f$. As $a\in {\rm End}(X)$, there exists an open neighborhood $V$ of $a$ such that $V\subset f^{-1}(U)$ and $|{\rm Bd}_X(V)|=1$. By the openness of $f$, $f(V)$ is an open neighborhood of $f(a)$ and $f(V)\subset U$.

To complete the proof, we need to show $|{\rm Bd}_X(f(V))|=1$. In order to show this fact, it suffices to show that for each point $y\in {\rm Bd}_X(f(V))$, there is an $x\in {\rm Bd}_X(V)$ such that $f(x)=y$.
Let $y\in {\rm Bd}_X(f(V))$. There exists a point sequence $\{y_n: n\in {\mathbb N}\}\subset f(V)$ such that $y_n\to y$ as $n\to \infty$. Take $x_n\in f^{-1}(y_n)\cap V$ for every $n\in {\mathbb N}$. Then there are a subsequence $\{x_{n_k}: k\in {\mathbb N}\}\subset \{x_n:n\in {\mathbb N}\}$ and a point $x\in \overline{V}$ such that $x_{n_k}\to x$ as $n_k\to \infty$. It follows from the continuity of $f$ that $f(x)=y$. Since $y\in {\rm Bd}_X(f(V))$ and $f$ is open, we have $x\in {\rm Bd}_X(V)$.
This implies that $|{\rm Bd}_X(f(V))|\leq |{\rm Bd}_X(V)|$. Hence $|{\rm Bd}_X(f(V))|=1$. Thus we complete the proof.
\end{proof}

\begin{lemma}\label{L3.3}
 Let $X$ be a Peano continuum having a free arc, which is neither an arc nor a circle, and let $f:X\to X$ be an open map. If $A$ is a free arc in $X$ containing in a maximal free open interval $J$, then $f(A)$ is also a free arc in $X$.
\end{lemma}

\begin{proof}
Notice that $f(J)$ is a connected open subset in $X$ by the continuity and openness of $f$. Furthermore, by Lemma~\ref{L3.1}, $f(J)$ contains no simple troid. Thus $f(J)$ is homeomorphic to $(0,1)$ or $[0,1)$ according to Lemma~\ref{L2.4}. Since $f(A)$ is a connected compact subset in $f(J)$, it is obvious that $f(A)$ is a free arc in $X$.
\end{proof}

Let $f:I\to I$ be a continuous map on an arc $I$. A point $c$ is called a {\it turning point} of $f$ if $f$ has a local extremum at $c$ (with respect to the nature ordering on $I$) and $c$ is in the interior of $I$.

\begin{lemma}\label{L3.4}
 Let $X$ be a Peano continuum having a free arc, which is neither an arc nor a circle. Suppose that there is no maximal free open interval of Type I in $X$. If $J$ is a maximal free open interval in $X$, and $f:X\to X$ is an open map, then there is a maximal free open interval $K$ such that $f(J)\subset K$, and $f: J\to f(J)\subset K$ is injective.
\end{lemma}

\begin{proof}
It is not difficult to check that $f(J)$ is a connected open subset in $X$ since $f$ is a continuous open map. According to Lemma~\ref{L3.1}, $f(J)$ contains no simple troid. Thus by Lemma~\ref{L2.4} and Lemma~\ref{L2.5}, there exists a maximal free open interval $K$ such that $f(J)\subset \overline{K}$. Since there is no maximal free open interval of Type I in $X$, we can show that $f(J)\subset K$.
To prove this fact, we distinguish two cases:

\noindent{\bf Case 1.}\quad $\overline{K}=K\cup\{c\}$ is a circle for $c\in X\setminus K$. If $c\in f(J)$, since $f(J)$ contains no simple troid, then $\overline{K}$ is open in $X$ by Lemma~\ref{L2.3}. Thus $X=\overline{K}$ by the connectedness of $X$, which is a contradiction. Hence $f(J)\subset K$.

\noindent{\bf Case 2.}\quad $\overline{K}=K\cup\{a,b\}$ is an arc for $a,b\in X\setminus K$ and $\{a,b\}\cap {\rm End}(X)=\emptyset$. If $a\in f(J)$, since $f(J)$ contains no simple troid, then according to Lemma~\ref{L2.3} and the maximality of $K$, $a\in {\rm End}(X)$, which is a contradiction. Hence $a\notin f(J)$. Similar arguments show that $b\notin f(J)$. Hence $f(J)\subset K$.

Now we prove that $f: J\to f(J)\subset K$ is injective.
Assume to the contrary that $f:J\to f(J)\subset K$ is not injective. Then there are two points $x_1, x_2\in J$ such that $f(x_1)=f(x_2)$. Thus there exists a turning point $z\in (x_1,x_2)_J$ such that $f(z)\in K$, which contradicts the openness of $f$. Thus $f:J\to f(J)$ is injective.
\end{proof}

\begin{lemma}\label{L3.5}
 Let $X$ be a Peano continuum having a free arc, which is neither an arc nor a circle. If $J$ is a maximal free open interval of Type I in $X$, and $f:X\to X$ is an open map, then there is a maximal free open interval $K$ of Type I in $X$ such that $f(J)\subset K$, and $f: J\to f(J)\subset K$ is injective.
\end{lemma}

\begin{proof} Since $J$ is a maximal free open interval of Type I, let $\overline{J}=J\cup\{u,v\}$, where $u,v\in X\setminus J$ and $u\in {\rm End}(X)$. For the same reason as in Lemma~\ref{L3.4}, there exists a maximal free open interval $K$ such that $f(J)\subset \overline{K}$. Since $u$ is an endpoint of $X$, then $f(u)\in {\rm End}(X)$ by Lemma~\ref{L3.2}. It implies that $K$ is a maximal free open interval of Type I in $X$. Let $\overline{K}=K\cup \{a,b\}$, where $a, b\in X\setminus K$ and $a=f(u)\in {\rm End}(X)$. In order to prove $f(J)\subset K$, it suffices to prove that $a,b\notin f(J)$. If not,
assume that $b\in f(J)$. According to Lemma~\ref{L2.3}, Lemma~\ref{L3.1} and the maximality of $K$, $b$ is an endpoint of $X$. It follows that $X=\overline{K}$, which is a contradiction. Therefore $b\notin f(J)$. On the other hand, if assume that $a\in f(J)$, which means there is a point $x\in J$ such that $f(x)=a=f(u)$, then $(x,u)_J$ contains a turning point $z$ such that $f(z)\in \overline{K}$. By the openness of $f$, we have $f(z)=b$.
Thus $b\in f(J)$. However, we have verified that $b\notin f(J)$. Hence $a\notin f(J)$ as well. Thus $f(J)\subset K$.

By the same argument as in the proof of Lemma~\ref{L3.4}, we have that
$f:J\to f(J)$ is injective. This completes the proof.
\end{proof}

\section{Proof of the main theorem}
In this section, we are ready to show the main theorem. First let us introduce some definitions.

Let $A$ be an arc. Suppose that $x_0,x_1,\cdots, x_m$ are points
in $A$ with ${\rm End}(A)=\{x_0,x_m\}$. If $x_0\prec x_1\prec
x_2\prec\cdots\prec x_{m-1}\prec x_m$ with respect to a natural
ordering $\prec$ on $A$, and $d(x_k,x_0)=k\cdot d(x_m,x_0)/m$ for
every $k=0,1,\cdots, m$, then the sequence
$\{x_1,x_2,\cdots,x_{m-1}\}$ is said to be {\it a set of
pseudo-m-section points of $A$ from $x_0$ to $x_m$}.

Let $X$ be a metric space with metric $d$. The collection of all non-empty compact subsets of $X$ is called to be the {\it hyperspace} of $X$ and denoted by $2^X$. For $A,B\in 2^X$ define the Hausdorff metric $d_H(A,B)$ between $A$ and $B$ by $d_H(A,B)={\rm inf}\{\epsilon >0: A\subset D_{\epsilon}(B) {\rm \ and\ }B\subset D_{\epsilon}(A)\}$. Now let us prove Theorem~\ref{T1.1}.

\begin{proof}[Proof of Theorem 1.1]
Assume to the contrary that $X$ is neither an arc, nor a circle. Let $c$ be a sensitive constant of $F$, and denote $\epsilon=c/7$. Let $\delta=\delta(\epsilon)$ be as the statement in Lemma~\ref{L2.1}. Since $X$ is compact and locally arcwise connected, there are finitely many arcwise connected open subsets $U_1,\cdots, U_n$ of $X$ such that $\cup_{i=1}^{n}U_i=X$ and ${\rm diam}(U_i)<\epsilon$ for each $1\leq i \leq n$.

For any set of open maps $\{f_1,\cdots, f_m\}\subset F$, define
$$S(f_1,\cdots,f_m)=\sum_{i=1}^{n}d(U_i, \cup_{j=1}^{m}f_j(U_i)) \eqno(4.1)$$

For an arc $A\subset X$, if $A\cap f_j(A)=\emptyset$ for all $1\leq j \leq m$, then $A$ is called to be a {\it jumping arc of $\{f_1,\cdots,f_m\}$}.

To show this theorem, we should prove some claims. The first claim is obvious.\vspace{1mm}

\noindent{\it Claim A.} If $A$ is a jumping arc of $\{f_1, \cdots, f_m\}$, then any subarc of $A$ is also a jumping arc of $\{f_1, \cdots, f_m\}$.\vspace{1mm}

\noindent{\it Claim B.} Let $A$ be a free arc in a maximal free open interval $J$. Suppose that $A$ is a jumping arc of $\{f_1,\cdots, f_m\}$, and $A,f_1(A),\cdots,f_m(A)$ are contained in the same maximal free open interval $J$. If $g: J\to g(J)$ is injective for $g\in F$, then $g(A)$ is also a jumping arc of $\{f_1,\cdots,f_m\}$.\vspace{1mm}

\noindent{\it Proof of Claim B.} By Lemma~\ref{L3.3}, $g(A)$ is a free arc in $X$. Since $F$ is a commutative semigroup and $g$ is injective on $J$, then we have
$$g(A)\cap f_j(g(A))=g(A)\cap g(f_j(A))=g(A\cap f_j(A))=\emptyset,\  {\rm for\ each\ }\ 1\leq j\leq m.$$
So $g(A)$ is a jumping arc of $\{f_1,\cdots,f_m\}$.\vspace{1mm}

 Applying the sensitivity of $F$ and Lemma~\ref{L3.3}, we have \vspace{1mm}

 \noindent{\it Claim C.} Let $K_0$ be a free arc in a maximal free open interval $J$. Then for every $i\in {\mathbb N}$, there exist $g_{i-1}, L_{i-1}, I_i, y_{i0},\cdots, y_{i7}$ and $K_i$ satisfying the following conditions:

 (i) $g_{i-1}\in F$, $L_{i-1}$ is a subarc of $K_{i-1}$, $I_i=g_{i-1}(L_{i-1})$ is a free arc in $X$, ${\rm End}(I_i)=\{y_{i0}, y_{i7}\}$, and $d(y_{i0}, y_{i7})=c=7\epsilon$.

 (ii) $\{y_{i1},\cdots, y_{i6}\}$ is the set of pesudo-7-section points of $I_i$ from $y_{i0}$ to $y_{i7}$, and $K_i=[y_{i3}, y_{i4}]_{I_i}$. \vspace{1mm}

In the following, we will discuss in two cases.

\noindent{\bf Case 1.} There is no maximal free open interval of Type I in $X$.

 In this case, according to Lemma~\ref{L3.4}, we see that for each maximal free open interval $J$ and each open map $g\in F$, $g:J\to g(J)$ is injective. Basing on this fact we can show the following two claims.\vspace{1mm}

\noindent{\it Claim D.} There exist a free arc $A_1$ of $X$ and an $f_1\in F$ such that $A_1$ is a jumping arc of $f_1$ with $A_1$ and $f_1(A_1)$ being contained in a same maximal free open interval.\vspace{1mm}

\noindent{\it Proof of Claim D.} Let $J$ be a maximal free open interval in $X$, and take a free arc $K_0\subset J$. According to Claim C, there exist $g_{i-1}$, $L_{i-1}$, $I_i$, $y_{i0},\cdots,y_{i7}$ and $K_i$ satisfying the conditions (i) and (ii) for every positive integer $i$. By the compactness of the hyperspace $(2^X, d_H)$, there exist integers $q>p>0$ such that
 $$d_H(K_p,K_q)<\delta.\eqno(4.2)$$
 By $(4.2)$ and Lemma~\ref{L2.2}, we have $K_p\subset (y_{q2},y_{q5})_{I_q}$. Let $f_1=g_{q-1}\cdots g_{p}$. Since $[y_{q0}, y_{q2}]_{I_q}\subset I_q\subset f_1(K_p)$, there is a free arc $A_1\subset K_p$ such that $f_1(A_1)=[y_{q0},y_{q2}]_{I_q}$. Thus $f_1(A_1)\cap K_p=\emptyset$, and so $f_1(A_1)\cap A_1=\emptyset$. That means $A_1$ is a jumping arc of $f_1$. By Lemma~\ref{L3.4}, there exist maximal free open intervals $J_p$ and $J_q$ such that $g_{p-1}\cdots g_0(J)\subset J_p$ and $g_{q-1}\cdots g_0(J)\subset J_q$. Since $K_p\subset J_p$ and $K_p\subset J_q$, then $J_p=J_q$. Obviously, $A_1\subset J_p$ and $f_1(A_1)\subset J_q$, which means $A_1$ and $f_1(A_1)$ are contained in the same maximal free open interval.\vspace{1mm}

 \noindent{\it Claim E.} Let $f_1,\cdots, f_m\in F$. If there is a free arc $A_m\subset X$ such that $A_m$ is a jumping arc of $\{f_1,\cdots,f_m\}$, and $A_m, f_1(A_m), \cdots, f_m(A_m)$ are contained in a same maximal free open interval, then there exist a free arc $A_{m+1}$ and $f_{m+1}\in F$ such that $A_{m+1}$ is a jumping arc of $\{f_1,\cdots, f_m,f_{m+1}\}$, and $A_{m+1}, f_1(A_{m+1}),\cdots,$ $f_{m+1}(A_{m+1})$ are contained in a same maximal free open interval. Moreover,
 $$S(f_1,\cdots,f_m,f_{m+1})<S(f_1,\cdots,f_m)-\delta.\eqno(4.3)$$

 \noindent{\it Proof of Claim E.} Let $K_0=A_m$. By Claim C, for each $i\in {\mathbb N}$, there are $g_{i-1}, L_{i-1}, I_i, y_{i0},\cdots y_{i7}$ and $K_i$ satisfying the conditions (i) and (ii). Similar to Claim D, there are integers $t>s>0$ such that $d_H(K_s,K_t)<\delta$. By Lemma~\ref{L2.2}, $K_s\subset (y_{t2},y_{t5})_{I_t}$. Let $f_{m+1}=g_{t-1}\cdots g_s$. Since $[y_{t0},y_{t2}]_{I_t}\subset f_{m+1}(K_s)$, there is a free arc $A_{m+1}\subset K_s$ such that $f_{m+1}(A_{m+1})=[y_{t0},y_{t2}]_{I_t}$. Thus $f_{m+1}(A_{m+1})\cap K_s=\emptyset$. So $f_{m+1}(A_{m+1})\cap A_{m+1}=\emptyset$. By Lamma~\ref{L3.4}, $g_{s-1}\cdots g_0$ is injective on $J$. Since $A_m$ is a jumping arc of $\{f_1,\cdots, f_m\}$, and $A_m, f_1(A_m), \cdots, f_m(A_m)$ are contained in a same maximal free open interval, then $g_{s-1}\cdots g_0(A_m)$ is also a jumping arc of $\{f_1,\cdots f_m\}$ by Claim B. Since $A_{m+1}$ is a subarc of $g_{s-1}\cdots g_0(A_m)$, then $A_{m+1}$ is also a jumping arc of $\{f_1,\cdots, f_m\}$ by Claim A. Hence, $A_{m+1}$ is a jumping arc of $\{f_1,\cdots,f_m, f_{m+1}\}$.

Now we prove that $A_{m+1}, f_1(A_{m+1}),\cdots, f_{m+1}(A_{m+1})$ are in a same maximal free open interval.
 Since $K_s\subset (y_{t2}, y_{t5})_{I_t}\subset I_t\subset f_{m+1}(L_s)\subset f_{m+1}(K_s)$, then there exists a fixed point $w\in K_s$ such that $f_{m+1}(w)=w$. Let $J$ be the maximal free open interval which contains $A_m, f_1(A_m), \cdots, f_m(A_m)$. By Lemma~\ref{L3.4}, there exist maximal free open intervals $J_s$ and $J_t$ such that $g_{s-1}\cdots g_{0}(J)\subset J_s$ and $g_{t-1}\cdots g_{0}(J)\subset J_t$. Since $w\in J_s$ and $w=f_{m+1}(w)\in J_t$, then $J_s= J_t$. Notice that $A_{m+1}\subset K_s\subset J_s$, $f_{m+1}(A_{m+1})\subset f_{m+1}(J_s)\subset J_t$. In addition, $f_i(A_{m+1})\subset f_i(g_{s-1}\cdots g_0(A_m))=g_{s-1}\cdots g_0(f_i(A_m))\subset g_{s-1}\cdots g_0(J)\subset J_s$ for every $i=1,\cdots,m$. Therefore $A_{m+1}, f_1(A_{m+1}), \cdots, f_{m+1}(A_{m+1})$ are contained in the same maximal free open interval $J_s$.

 In the following, we show $(4.3)$ holds. Take $k\in \{1,2,\cdots,n\}$ such that $w\in U_k$. Then $U_k\cap f_{m+1}(U_k)\neq\emptyset$. So $d(U_k, f_{m+1}(U_k))=0$. On the other hand, Since $U_k$ is arcwise connected such that ${\rm diam}(U_k)<\epsilon$ and $U_k\cap (y_{t2}, y_{t5})_{I_t}\neq\emptyset$, then $U_k\subset [y_{t1},y_{t6}]_{I_t}\subset I_t$. Notice that $I_t\subset g_{t-1}\cdots g_0(A_m)$, and $g_{t-1}\cdots g_0$ is injective on $J$, then by Claim A and Claim B, $I_t$ is a jumping arc of $\{f_1,\cdots, f_m\}$. So $\cup_{j=1}^{m}f_j(U_k)\subset \cup_{j=1}^{m}f_j(I_t)\subset X-I_t$. Then it follows from Lemma~\ref{L2.2} that $d(U_k, \cup_{j=1}^{m}f_j(U_k))>\delta$. Hence,
\begin{align*}
& S(f_1,\cdots,f_m)-S(f_1,\cdots, f_m,f_{m+1})\\
=& \sum_{i=1}^{n} d(U_i,\cup_{j=1}^{m}f_j(U_i))-\sum_{i=1}^{n} d(U_i,\cup_{j=1}^{m+1}f_j(U_i))\\
=& \sum_{i=1}^{n}\Big[d(U_i,\cup_{j=1}^{m}f_j(U_i))-d(U_i,\cup_{j=1}^{m+1}f_j(U_i))\Big]\\
\geq & d(U_k,\cup_{j=1}^{m}f_j(U_k))-d(U_k,\cup_{j=1}^{m+1}f_j(U_k))\\
=& d(U_k,\cup_{j=1}^{m}f_j(U_k))\\
>& \delta.
\end{align*}
Thus the proof of Claim E is complete.\vspace{1mm}

\noindent{\bf Case 2.} There exist maximal free open intervals of Type I in $X$.

In this case, we can also have two claims similar to Claim D and Claim E. Notice that a key point in the proofs of Claim D and Claim E is the injectivity of $g$ on $J$ for any $g\in F$ and any maximal free open interval $J$ in $X$. In this case, we begin our discussion with choosing a maximal free open interval $J$ of Type I and taking a free arc $K_0\subset J$. According to Lemma~\ref{L3.5}, for any $g\in F$ and any maximal free open interval $J$ of Type I, $g:J\to g(J)$ is injective. Moreover, $g(J)$ is also contained in a maximal free open interval of Type I. So we can ensure that all free arcs we discussed in the following proofs are contained in maximal free open intervals of Type I, and all open maps are injective on any maximal free open interval of Type I. Thus using arguments similar to those in Case 1, we have the following claims.\vspace{1mm}

\noindent{\it Claim D$'$.} There exist a free arc $A_1$ of $X$ and an $f_1\in F$ such that $A_1$ is a jumping arc of $f_1$ with $A_1$ and $f_1(A_1)$ being contained in a same maximal free open interval of Type I.\vspace{1mm}

 \noindent{\it Claim E$'$.} Let $f_1,\cdots, f_m\in F$. If there is a free arc $A_m\subset X$ such that $A_m$ is a jumping arc of $\{f_1,\cdots,f_m\}$, and $A_m, f_1(A_m), \cdots, f_m(A_m)$ are contained in a same maximal free open interval of Type I, then there exist a free arc $A_{m+1}$ and $f_{m+1}\in F$ such that $A_{m+1}$ is a jumping arc of $\{f_1,\cdots, f_m,f_{m+1}\}$, and $A_{m+1}, f_1(A_{m+1}),\cdots, f_{m+1}(A_{m+1})$ are contained in a same maximal free open interval of Type I. Moreover,
 $$S(f_1,\cdots,f_m,f_{m+1})<S(f_1,\cdots,f_m)-\delta.$$

Finally, let us finish the proof of this theorem. Denote ${\rm diam}(X)=r$. Then for any $f\in F$, we have $S(f)\leq nr$. Take an integer $l>nr/\delta$. By Claim D and Claim E (for Case 1), or by Claim D$'$ and Claim E$'$ (for Case 2), there are $f_1,f_2,\cdots, f_{l+1}\in F$ such that $S(f_1,\cdots, f_j,f_{j+1})<S(f_1,\cdots, f_j)-\delta$ for all $1\leq j\leq l$. This implies that
$$S(f_1,\cdots, f_l,f_{l+1})<S(f_1)-l\delta\leq nr-l\delta<0,$$
However, it follows from the definition of (4.1) that $S(f_1,\cdots,f_l,f_{l+1})\geq 0$, which is a contradiction. The proof is complete.
\end{proof}

\noindent{\Large\bf Acknowledgements}\vspace{2mm}

The first author is supported by NSFC (No. 11401263) and the second author is supported by NSFC (No. 11771318, No. 11790274).

\end{document}